\documentclass{article}
\usepackage{amsmath,amssymb,amsthm,fullpage}
\usepackage{tabu}

\newtheorem{theorem}{Theorem}
\newtheorem{lemma}{Lemma}
\newtheorem{corollary}{Corollary}
\newtheorem{pretheorema}{{\bf Theorem}}

\newenvironment{theorema}[1]{\begin{pretheorema}
{\hspace{-0.2em}{\rm #1}{\bf}}}{\end{pretheorema}}

\def\sud{{\mathrm{Sud}}}

\title{{Sudoku Rectangle Completion}}
\author{Mohammad Mahdian$^1$ \and Ebadollah S. Mahmoodian$^2$\thanks{This research is partially  supported by a grant from the INSF.}}
\begin{document}
\date{}
\maketitle
\vspace*{-1cm}
\begin{center}
\footnotesize{
$^1$Google Research \\
Mountain View, CA, USA\\[3mm]
$^2$Department of Mathematical Sciences \\
 Sharif University of Technology \\
 Tehran, I.R. IRAN \\
}
 \end{center}

\begin{abstract}
Over the last decade, Sudoku, a combinatorial number-placement puzzle,
has become a favorite pastimes of many all around the world. In this
puzzle, the task is to complete a partially filled $9 \times 9$ square with
numbers 1 through 9, subject to the constraint that each number must
appear once in each row, each column, and each of the nine $3 \times 3$
blocks. Sudoku squares can be considered a subclass of the
well-studied class of Latin squares. In this paper, we study natural
extensions of a classical result on Latin square completion to Sudoku
squares. Furthermore, we use the procedure developed in the proof 
to obtain asymptotic bounds on the number of Sudoku squares of order $n$.
\end{abstract}

\section{Introduction and preliminaries}
\label{sec:definitions}
A Latin square is an $n\times n$ matrix with entries in $1,\ldots,n$  such that each of the numbers $1$ to $n$ appears exactly once in each row and in each column. Latin squares are heavily studied combinatorial objects that date back to the time of Euler and probably earlier. 
A variant of the notion of Latin squares has recently surfaced in the form of a number-placement
puzzle called Sudoku. In this puzzle, the task is to complete a partially filled $9 \times 9$ square with
numbers 1 through 9 such that in addition to the Latin square conditions, each number appears exactly 
once in each of the nine $3 \times 3$ blocks. This puzzle was popularized in 1986 by the Japanese puzzle
company Nikoli and became an international hit in the 2000s, although similar puzzles have appeared in 
various publications around the world since late 19th century. The emergence of this puzzle has generated
a surge of interest in the mathematical properties of Sudoku squares~\cite{MR2882508}.

In this paper, we study a Sudoku rectangle completion problem similar to a classical result of M. Hall on 
Latin rectangle completion. 
We start with the formal definition of the main notions used in this paper. 
Definitions and notations not given here may be found in standard combinatorics and graph theory
textbooks such as~\cite{MR2368647} and \cite{MR1871828}.

A {\sf Latin square} of order $n$ is an $n\times n$ matrix with entries from
$[n]=\{1,\ldots,n\}$ that satisfies the following two conditions:

\begin{itemize}
\item {\bf row condition}: every element in $[n]$ appears at most once in 
each row.
\item {\bf column condition}: every element in $[n]$ appears at most once 
in each column.
\end{itemize}

When $n = k^2$ for an integer $k$, for every
$i,j\in[k]$, the {\sf $(i,j)$th block} of an $n\times n$ matrix $M$ is defined as 
the set of entries with coordinates in $((i-1)k + x,
(j-1)k + y)$ for $x,y\in [k]$. We say that $(i,j)$ are the coordinates
of this block. These blocks partition the set of entries in $M$ into $k^2$ submatrices, 
each of size $k\times k$ and therefore containing $n$ entries. 
The $i$th {\sf row block} of $M$ is the union of the blocks at coordinates $(i,j)$ 
for $j\in[k]$. Similarly, the $j$th {\sf column block} of $M$ is the union of the blocks at coordinates $(i,j)$ 
for $i\in[k]$.
A {\sf Sudoku square}
of order $n=k^2$ is an $n\times n$ matrix that in addition to the row and 
column conditions above, satisfies the following condition:

\begin{itemize}
\item {\bf block condition}: every element in $[n]$ appears at most once in 
each block.
\end{itemize}

A {\sf partial Latin (Sudoku) square} of order $n$ is an $n\times n$ matrix with
entries from $[n]\cup\{*\}$ (with $*$ representing an {\em empty} entry) 
that satisfies the row and column (row, column, and block) conditions. A partial Latin (Sudoku) square
$P_2$ is an {\sf extension of a partial Latin (Sudoku) square} $P_1$ if they have the same
order and for every
$(i,j)\in[n]^2$, if $P_1(i,j)\neq *$, then $P_1(i,j)=P_2(i,j)$. For $m
< n$, an $m\times n$ Latin (Sudoku) rectangle is an $n\times n$ partial Latin (Sudoku)
square in which all cells in the first $m$ rows are filled and all
remaining cells are empty. More generally, for every $p, q\le n$, a
$(p,q,n)$-Latin (Sudoku) rectangle is an $n\times n$ partial Latin (Sudoku) square
in which all cells in the intersection of the first $p$ rows and the
first $q$ columns are filled and all remaining cells are empty.

For Latin rectangles there is a well-known theorem of Marshal Hall~\cite{MR0013111} that states:
Every $m\times n$ Latin rectangle can be extended to an $n\times n$
Latin square. This theorem is proved by using the classical matching theorem of Philip Hall. 
A natural question is whether the similar statement holds for Sudoku rectangles.
In this paper, we study this question, and will show that, perhaps surprisingly, the answer
depends on the value of $m$.

For $n=9$ (the regular Sudoku), this question was answered by
Kanaana and Ravikumar~\cite{MR2654731}. They showed that for all values
of $m$ except $m=5$, an $m\times 9$ Sudoku rectangle can always be 
completed to a Sudoku square. For $m=5$, this is not always possible. For 
example, see the $5\times 9$ Sudoku rectangle in Figure~\ref{fig:example}.
Note that none of the numbers $1,\ldots,9$ can be placed in the square marked
with a star. 

\begin{figure}[h]
\centerline{
\begin{tabu}{|[2pt]c|c|c|[2pt]c|c|c|[2pt]c|c|c|[2pt]}
\tabucline[2pt]{-}
  1 & 2 & 3 & 4 & 5 & 6 & 7 & 8 & 9\\\hline
  4 & 5 & 6 & 7 & 8 & 9 & 1 & 2 & 3\\\hline
  7 & 8 & 9 & 1 & 2 & 3 & 4 & 5 & 6\\\tabucline[2pt]{-}
  8 & 3 & 2 & 5 & 6 & 1 & 9 & 4 & 7\\\hline
  9 & 6 & 5 & 8 & 4 & 7 & 2 & 3 & 1\\\hline
  $\star$ & & & & & & & & 
\end{tabu}
}
\caption{A $5\times 9$ Sudoku rectangle with no valid completion}
\label{fig:example}
\end{figure}

For general $n$, the only previous result on completability of Sudoku rectangle 
is the following theorem proved by Kanaana and Ravikumar~\cite{MR2654731}.
 \begin{theorema}{ \rm (\cite{MR2654731})}
\label{thm:lastblock}
Assume $n=k^2$ and $n-k\le m< n$. Then every $m\times n$ Sudoku
rectangle can be completed to a Sudoku square.
\end{theorema}

In this paper, we prove two more sufficient conditions
for $m$, under which every $m\times n$ Sudoku rectangle is completable
to an $n \times n$ Sudoku square (Section~\ref{sec:sufficient}). 
These results are proved using a two-stage procedure for
completing a Sudoku rectangle, where generalized matching and
bipartite graph edge coloring is used within the two stages. We will show in
Section~\ref{sec:characterization} that the union of these  three conditions is a
full characterization of the values of $m$ for which every $m\times n$
Sudoku rectangle is completable. An important ingredient of this proof is a constructive lemma
that shows how to generalize a non-completable $m \times k$ Sudoku rectangle to 
a non-completable $m\times n$ Sudoku rectangle.
Sections~\ref{sec:alg} and~\ref{sec:counting} are devoted to two corollaries of our characterization.
We observe in Section~\ref{sec:alg}, that the procedure used to prove our sufficient conditions
gives an efficient
algorithm for deciding whether a given $m\times n$ rectangle is
completable and find a valid completion. Furthermore, in
Section~\ref{sec:counting} we use this procedure combined with
theorems of Minc and Van der Waerden, to prove asymptotic bounds on
the number of Sudoku squares of order $n$.

\paragraph{Related work.}
Recently quite a few papers have appeared on the mathematical questions raised for Sudoku squares 
similar to Latin squares. For example in ~\cite{FelgenhauerJarvis06} the number of Sudoku squares is 
discussed, while some results about the sets of mutually orthogonal Sudoku 
squares are given in~\cite{PedersenVis09}.  Concepts of ``critical (or defining) sets''  i.e. 
``the minimum Sudoku problem''  and 
``Sudoku trades (or detection of unavoidable sets in Sudoku)''
are studied in~\cite{LinWu10}, \cite{MR3223774}, ~\cite{Williams11}, and~\cite{MR3130362}. 
In~\cite{MR2408485} Sudoku is considered as a special case of Gerechte designs and they introduce some 
interesting mathematical problems about them. The complexity and completeness of finding 
solution for Sudoku is investigated in~\cite{TakayukiTakahiro03}. List coloring of Latin and 
Sudoku graphs is studied in~\cite{MR2882508} which has an extensive number of 
references on Sudoku. 
In~\cite{MR2654731} the problem of completing Sudoku rectangles is studied.

\section{Sufficient Conditions}
\label{sec:sufficient}
In Theorem~\ref{thm:lastblock} it is shown 
that if the Sudoku rectangle only leaves parts of the last row block
incomplete, then it can be completed to a Sudoku square.
In this section, we prove two theorems, showing that for
certain other values of $m$, an $m\times n$ Sudoku rectangle can always be
completed to a Sudoku square. We start with the  next result which shows that
if the Sudoku rectangle consists of a number of full row blocks and a
number of empty row blocks (i.e., no row block is partially filled), then it
can always be completed to a Sudoku square.

\begin{theorem}
\label{thm:fullblocks}
Assume $n=k^2$ and $m = lk$ for some $l< k$. Then every $m\times n$
Sudoku rectangle $R$ can be completed to a Sudoku square.
\end{theorem}
\begin{proof}
We prove that we can complete the rectangle $R$ row-block by
row-block. To do this, we need to show that $R$ can be extended to a
$k(l+1) \times n$ Sudoku rectangle. This is done in two steps: First,
we fill the elements in the $(l+1)$st row block in such a way that column 
and block conditions are satisfied. In the second step, we permute the
elements in each column of the $(l+1)$st row block in such a way that the resulting
configuration satisfies the row condition as well. Note that column and block conditions
will stay satisfied after such a permutation, and therefore, the resulting configuration
is a valid Sudoku extension of $R$.

We start with the first step. The $(l+1)$st row block consists of $k$
blocks, which we denote by $B_1, B_2, \ldots, B_k$. For each block
$B_d$, we formulate an assignment problem as follows: On one side of
the assignment problem, we have the numbers $1,\ldots,n$. On the other
side, we have the $k$ columns $c_1, c_2, \ldots  , c_k$ of the block. A number $i$ can be
assigned to a column $j$, if $i$ does not appear in  column $j$ in
$R$. For each number $i$, this number appears exactly once in each of
the blocks above $B_d$, and these appearances are in different
columns. Therefore, each number $i$ can be assigned to exactly $k-l$
different columns. Also, for each column $j$, there are precisely $ lk$
distinct numbers in that column in $R$. Therefore, there are precisely
$n- lk$ numbers that can be assigned to column $j$. The objective is to
find an assignment that assigns each number to exactly one column, and
assigns precisely $k$ numbers to each column.

We show that this assignment problem always has a feasible
solution. We prove this by using a network flow argument. The assignment
problem can be formulated as a network flow problem: There is a source
$s$ that has links of capacity 1 to $n$ nodes representing the
numbers; each number $i$ has a link of capacity 1 to each of the nodes representing
columns that it can be assigned to; and there is a sink $t$ that has a
link of capacity $k$ from each of the nodes representing the
columns. It is easy to see that it is possible to route a flow of
value $n$ from $s$ to $t$ in this network: The flow on each of the
edges from $s$ to number nodes is 1; the flow on each of the edges
from number nodes to column nodes is $1/(k-l)$; and the flow from each
of the column nodes to $t$ is $k$. It is not hard to see that this is
a feasible flow of value $k$ (and is therefore the maximum flow from
$s$ to $t$). By the Ford-Fulkerson theorem, there must be an integer
flow of the same value in this network. This integer flow gives an
assignment of numbers $1,\ldots,n$ to the $k$ columns in $B_d$ such
that each number is assigned to precisely one column, each column has
precisely $k$ numbers assigned to it, and no number is assigned to a
column $c_i$ where it appears in the $i$th column of $R$. Putting
these assignments together for all blocks $B_d$ of the $(l+1)$st row block, we obtain an assignment
of the $n$ numbers to the $n$ columns such that:
\begin{itemize}
\item each number is assigned to precisely $k$ columns, one from each block;
\item each column has precisely $k$ numbers assigned to it; and
\item no number is assigned to a column where it appears in one of the cells of the column
  above that.
\end{itemize}

Let us call this assignment $M$. The second step is to take the
assignment $M$ and specify in which of the $k$ rows of the $(l+1)$st row block
each number must go. We construct a bipartite graph
as follows: On one side, we have the numbers $1,\ldots,n$, and on the
other side, we have the $n$ columns of the $(l+1)$st row block. There is an edge between number
$i$ and column $j$, if $i$ is assigned to $j$ under the assignment
$M$. By the properties of $M$, this is a bipartite $k$-regular
graph. By K\"{o}nig's theorem, this graph has an edge-coloring with
$k$ colors. Let the colors be denoted by $1,\ldots, k$. We now
complete the $(l+1)$st row block as follows: For each edge $(i,j)$
colored with color $c$, we place the number $i$ on the $c$th row of the
row block and $j$th column. It is easy to see that with the addition of
this row-block all the Sudoku conditions are still satisfied: There is
no repeated element in any row, since the coloring is a proper edge
coloring; there is no repeated element in each column by the third
property of the assignment $M$; and there is no repeated element in
each block by the first property of the assignment $M$.
\end{proof}

Next, we prove the following sufficient condition for completability
of a Sudoku rectangle.

\begin{theorem}
\label{thm:thirdcase}
Assume $n=k^2$ and $m =  lk+r$ for some $l< k$ and $0 \le r <k$. Then an $m\times n$
Sudoku rectangle $R$ can always be completed to a Sudoku square, if $(k-r)(k-l)\ge  lk$.
\end{theorem}
\begin{proof}
We proceed the same way as in the proof of Theorem \ref{thm:fullblocks}. The empty rows of the $(l+1)$st 
row block form a $(k-r) \times  n$  rectangle $R_1$. First, we show
that for each column in $R_1$, we can pick $k-r$ distinct elements in such a
way that the elements chosen for the columns in each block are
distinct, and the $k-r$ elements for each column do not appear in that
column in $R$. Next, we show that the elements in each column can be permuted
in such a way that the elements in each row become distinct. Note that
for the second step exactly the same proof as in Theorem~\ref{thm:fullblocks}
works. Therefore, we only need to prove the first step. Also, note
that the first step can be done block-by-block. So, it is enough to
consider one block $B_d$, and prove the following: 
for each column in $B_d$ we
can pick $k-r$ elements distinct from elements that appear in $B_d$ and in the cells of the same column above $B_d$ in such a way that the elements picked for the $k$ columns in
$B_d$ are distinct.

We formulate the problem as a matching problem as follows: On one
side, we have $k$ vertices corresponding to the $k$ columns in the
block $B_d$. We call these the column vertices. On the other side, we have
$n-rk$ vertices, corresponding to the elements that do not appear in
the $r$ filled rows of the block. We call these the element
vertices. There is an edge between a column vertex and an element
vertex if the corresponding element does not appear in the
corresponding column in $R$. The degree of each element vertex is
precisely $k-l$, since out of the $k$ columns, it appears in $l$ of
them (one for each block above this block). The degree of each column
vertex is a number between $n-lk-rk$ and $n-lk$. We prove that there
is a 1-to-$(k-r)$ matching in this graph, i.e., a matching which
matches each column vertex to precisely $k-r$ element vertices, and
each element vertex to precisely one column vertex.

We prove this by verifying the following easy extension of Hall's condition~\cite{MR1871828}
for the existence of a 1 to $(k-r)$ matching. 
This condition states that such a matching exists
if and only if for every set $S$ of column vertices, $|N(S)|\ge
(k-r)|S|$, where $N(S)$ is the set of all element vertices adjacent to
at least one vertex in $S$.  To prove this, we consider two cases:
either $|S|>l$ or $|S| \le l$. In the first case, consider an element
$x$ that does not appear in the $r$ filled rows of the block. Since
$x$ appears in exactly $l$ columns in the corresponding column block in $R$ and
$|S|>l$, there is at least one column in $S$ that does not contain $x$. 
Therefore, the corresponding column vertex is adjacent to the element vertex $x$. This means that in this case,
$N(S)$ is the set of all $n-rk$ element vertices. Hence,
$|N(S)|=k(k-r)\ge|S|(k-r)$, as desired. In the second case ($|S| \le
l$), we use the fact that the degree of each column vertex in this
graph is at least $n-lk-rk$, and therefore, $|N(S)|\ge n-lk-rk$. Using
this inequality, we have:

$$|N(S)|\ge k^2-lk-rk = (k-r)(k-l) - rl \ge  lk - rl \ge (k-r)|S|.$$

This completes the proof of Hall's condition, which implies that to
each column we can assign $k-r$ elements distinct from elements that
appear in the $r$ filled rows of the block, such that each such
element is assigned precisely to one column.
\end{proof}

\section{A Complete Characterization}
\label{sec:characterization}
In this section, we give some Sudoku rectangle constructions to prove
that the union of the three sufficient conditions given in
Section~\ref{sec:sufficient} and Theorem \ref{thm:lastblock} is also
necessary.

\begin{theorem}
\label{thm:main}
Assume $n=k^2$ and $m =  lk+r$ for some $l< k$ and $0 \le r <k$. Every
$m\times n$ Sudoku rectangle $R$ can always be completed to a Sudoku
square, if and only if at least one of the following conditions hold:
\begin{itemize}
\item $l = k-1$,
\item $r = 0$, or
\item $(k-r)(k-l)\ge lk$.
\end{itemize}
\end{theorem}

We proved sufficiency of the above conditions in the previous section. 
To prove necessity, we need to construct an $m\times n$ Sudoku
rectangle that is not completable for any value of $m$ that does not
satisfy the conditions of the theorem. The main tool we use in this
construction is the following extension of
Theorem~\ref{thm:fullblocks}, which proves that any 
$(m, k, n)$-Sudoku rectangle can be extended to an $m\times n$ Sudoku
rectangle. This means that to construct an $m\times n$ Sudoku
rectangle that is not completable, it is enough to construct its first
column block, i.e., an $(m, k, n)$-Sudoku rectangle that is not
completable.

\begin{lemma}
\label{lem:fullblocksgen}
For every $n=k^2$, $m\le n$, every $(m, k, n)$-Sudoku rectangle
$R$ can be extended to an $(m, n, n)$-Sudoku rectangle.
\end{lemma}
\begin{proof}
We follow an algorithm similar to the one used in the proof of
Theorem~\ref{thm:fullblocks} to extend $R$, column block by column
block. Before starting this algorithm, first we append a few rows to $R$ to
make sure the number of rows is a multiple of $k$: Let $m=lk + r$,
where $0\le r< k$. We append $k-r$ rows to $R$, containing the $n-rk$
elements that do not appear in rows $lk + 1$ through $lk + r$ in $R$
in an arbitrary order. This converts $R$ into an $(l+1)k\times k$
rectangle $R'$ that satisfies the row and block conditions, but might
violate the column condition (i.e., $R'$ is not necessarily a partial
Sudoku square). Now, we append a new column block to this rectangle
using a two-stage process as follows: In the first stage, for each row
in $R'$, we are going to pick $k$ elements among elements that do not already
appear in this row in $R'$, such that sets picked for rows within the
same row block are disjoint. This is done by constructing a graph
similar to the one constructed in the proof of
Theorem~\ref{thm:fullblocks} for each row block of $R'$, and proving
it contains a $k$-to-$1$ matching. The graph contains $k$ row vertices
corresponding to the rows within this row block, and $n$ element
vertices corresponding to the elements $1$ to $n$. There is an edge
between row vertex $i$ and element vertex $j$ if $j$ does not appear
in the $i$th row of $R'$. Since $R'$ satisfies row and block constraints, the
degree of each row vertex in this graph is precisely $n-k$ and the
degree of each element vertex is $k-1$. Therefore, using essentially
the same argument as in the proof of Theorem~\ref{thm:fullblocks}, we
can show that this graph contains a matching that matches each row
vertex to precisely $k$ element vertices and each element vertex to
exactly one row vertex. In the second stage, we assign each of the
$k$ elements assigned to each row to one of the $k$ columns within the
new column block, in such a way that the column condition is
satisfied. Again, this is done by building a bipartite graph as in the
proof of Theorem~\ref{thm:fullblocks}. The graph contains $(l+1)k$ row
vertices, corresponding to the rows of $R'$, and $n$ element
vertices. There is an edge between an element vertex $i$ and a row
vertex $j$ if the element $i$ is assigned to the row $j$ in the first
stage. The maximum degree of vertices in this graph is $k$. Therefore,
the graph is $k$-edge colorable. We assign each color to one column of
the new column block, and place an element $i$ in row $j$ in column
$c$ of this column block, if the edge between $i$ and $j$ is colored
with the color corresponding to $c$. This ensures that the vertices
assigned to the same column are distinct. Therefore, we obtain an
$(l+1)k\times 2k$ extension of $R'$ that satisfies block and row
conditions, and satisfies the column conditions except possibly for
the first $k$ columns. We can continue this process, column block by
column block, resulting in an $(l+1)k\times n$ extension of $R'$ that
satisfies row and block conditions and the column conditions except
possibly for the first $k$ columns (i.e., columns of $R'$). Removing
the last $k-r$ rows of this rectangle results in an $m\times n$
rectangle that satisfies row, block, and column conditions.
\end{proof}

The following lemma is another useful tool in our constructions.

\begin{lemma}
\label{lem:2}
For every $a, b\le k$ and every set $A$ of $k\cdot\max\{a,b\}$ elements, there is
an $(ak, b, n)$-Sudoku rectangle with elements from $A$.
\end{lemma}
\begin{proof}
Let $c = \max\{a,b\}$.  We partition the elements of $A$ into $c$ sets
$A_0,A_1,\ldots,A_{c-1}$, each of size $k$. The Sudoku rectangle is
constructed as follows: We fill the elements in column $j + 1$ ($0\le
j< b$) and rows $ik + 1$ through $ik + k$ ($0\le i< a$) of the
rectangle with elements from $A_{(i + j)\mod c}$, in increasing
order\footnote{For the purpose of this lemma, elements of
  $A_{(i+j)\mod c}$ can be placed in any arbitrary order. However,
  placing them in increasing order makes it easier to combine and
  modify these rectangles, as needed in the proof of
  Theorem~\ref{thm:main}.}. Since for any fixed $i$, the family
$\{A_{(i + j)\mod c}: 1\le j\le b\}$ is a family of disjoint sets, the
block conditions (and therefore the row conditions) are
satisfied. Also, for every fixed $j$, the family $\{A_{(i + j)\mod c}:
1\le i\le a\}$ is a family of disjoint sets, and therefore the column
conditions are also satisfied.
\end{proof}

Equipped with the above lemmas, we are now ready to prove
Theorem~\ref{thm:main}:

\begin{proof}[Proof of Theorem~\ref{thm:main}]
The ``if'' part follows from Theorems~\ref{thm:lastblock},
\ref{thm:fullblocks}, and \ref{thm:thirdcase}. To prove the ``only
if'' part, let $m=lk + r$ be a number that does not satisfy any of the three conditions
(i.e., $l\le k-2$, $r \ge 1$, and $(k-r)(k-l)<lk$), and construct an
$m\times n$ Sudoku rectangle that cannot be completed to a Sudoku
square. We use Lemma~\ref{lem:fullblocksgen} in this construction as
follows: Instead of constructing an $m\times n$ Sudoku rectangle, we
construct an $(m, k, n)$-Sudoku rectangle that cannot be completed
into a Sudoku square, and use Lemma~\ref{lem:fullblocksgen} to extend
this $m\times k$ rectangle to an $m\times n$ rectangle. As the
original $m\times k$ rectangle was not completable, the $m\times n$
rectangle cannot be completable either.

We consider three cases for this construction: (a) $l < k/2$, (b)
$l\ge k/2$ and $k$ is even, and (c) $l\ge k/2$ and $k$ is odd.

\renewcommand{\labelenumi}{(\alph{enumi})}
\begin{enumerate}
\item $l < k/2$: Using Lemma~\ref{lem:2} with $a=b=l$, we can
  construct an $(lk, l, n)$-Sudoku rectangle $R_1$ with entries in
  $\{1,\ldots, lk\}$. Also, since $l + 1 \le k - l$, using
  Lemma~\ref{lem:2} with $a = l + 1$ and $b = k-l$ we obtain an $((l +
  1)k, k - l, n)$-Sudoku rectangle $R_2$ with elements in $\{lk
  + 1,\ldots, k^2\}$. The idea is to construct an $(m, k, n)$-Sudoku rectangle by concatenating the columns of $R_2$ after the
  columns of $R_1$. However, $R_1$ has $m - r$ rows and $R_2$ has $m +
  k - r$ rows, and therefore before concatenating $R_1$ and $R_2$, we
  need to add $r$ rows to $R_1$ and remove $k-r$ rows from $R_2$ to
  make sure both rectangles have $m$ rows. To do this, we first simply
  remove the extra $k - r$ rows from $R_2$. These rows contain a set
  of $(k-r)(k-l)$ elements $E$. We place these elements in an
  arbitrary order in the missing $r$ rows of $R_1$. To fill these rows
  with elements in $E$, we need $lr \le |E| = (k-r)(k-l) = lr + k(k -
  l - r)$. This holds if and only if $l + r \le k$, in which case we
  are done. If $l + r > k$, $R_1$ is still missing $lr - |E| = k(l + r
  - k)$ elements to become an $m\times l$ rectangle. We pick a set
  $E'$ of this many elements arbitrarily from rows $lk + 1$ through
  $lk + r$ of $R_2$. This is feasible, since $R_2$ has $(k-l)r$
  elements in these rows, and

$$(k-l)r = (r+l)k - (r+k)l > (r+l)k - (k+k)l > (r+l)k - 2k\cdot\frac{k}2 = k(l + r - k).$$

We now perform the following switch: The elements of $E'$ in the rows
$lk + 1$ through $lk + r$ of $R_2$ are replaced with arbitrary
elements in $1,\ldots, lk\}$ (possible since $k(l + r - k)\le  lk$), and
the missing elements of $R_1$ are filled with elements in $E'$.

It is easy to see that the rectangle $R$ obtained by concatenating the
columns of $R_1$ and $R_2$ after the above operations is an $(m,
k, n)$-Sudoku rectangle: The block conditions are satisfied since
$R_1$ and $R_2$ initially satisfied these conditions, they were
composed of disjoint sets of elements, and none of the steps above
violated the block condition. The row conditions follows directly from
the block conditions, and the column conditions are also true since
they are true for $R_1$ and $R_2$ constructed by Lemma~\ref{lem:2},
and the above operations only moved elements from $lk+1,\ldots, k^2$
to the first $l$ columns (which did not previously contain any of
these elements), and elements from $1,\ldots, lk$ to the last $k-l$
columns (which, again, did not previously contain any of these
elements).

Now, we prove that $R$ is not completable to a Sudoku square. We
consider two cases: $l+r \le k$ and $l + r > k$. In the latter case,
it is evident from the above construction that all elements in
$\{lk+1,\ldots,k^2\}$ appear in the last partial block (i.e., last $r$
rows) of $R$. Also, all elements in $\{1,\ldots, lk\}$ appear in each
of the first $l$ columns. Therefore, there is no element that can be
placed in any of the entries in columns $1$ through $l$ of the
$(m+1)$st row without violating column or block conditions.

We now consider the case $l + r \le k$. Consider the first $l$ columns
of $R$. To complete $R$, we must fill rows $m+1,\ldots,m+k-r = (l+1)k$
of these columns with $l(k-r)$ elements. Let us call the set of these
elements $S$. By block conditions, elements in $S$ must all be
distinct, and also distinct from the $rk$ elements in the first $r$
rows of this block. Furthermore, since all elements in
$\{1,\ldots, lk\}$ appear in each of the first $l$ columns, by column
conditions $S$ cannot contain any of the elements
$1,\ldots, lk$. Therefore, we must have at least $l(k - r) + rk +  lk$
distinct elements. Thus, $l(k-r) + rk + lk \le k^2$, or $lk\le k^2 -
lk - rk + lr = (k-l)(k-r)$, contradicting the third condition.

\item $l\ge k/2$ and $k$ is even: Let $F = \{1,\ldots,\frac{k^2}2\}$
  and $G = \{\frac{k^2}2 + 1,\ldots,k^2\}$. We use Lemma~\ref{lem:2}
  to construct the following four Sudoku rectangles:
\begin{itemize}
\item a $(\frac{k^2}{2}, \frac{k}2, n)$-Sudoku rectangle $R_1$ with
  entries in $F$,
\item a $(\frac{k^2}{2}, \frac{k}2, n)$-Sudoku rectangle $R_2$
  with entries in $G$,
\item an $((l+1-\frac{k}2)k, \frac{k}2, n)$-Sudoku rectangle $R_3$
  with entries in $G$, and
\item an $((l+1-\frac{k}2)k, \frac{k}2, n)$-Sudoku rectangle $R_4$
  with entries in $F$.
\end{itemize}

Recall that the construction in Lemma \ref{lem:2} starts by defining a
partition of the set of elements into subsets of size $k$. Let
$F_0,F_1,\ldots,F_{k/2-1}$ be the partition of $F$ used to construct
$R_1$, and $G_0,G_1,\ldots, G_{k / 2-1}$ be the partition of $G$ used
to construct $R_2$. To construct $R_3$, we use this permutation of the
latter partition: $G_1, G_3, G_4, \ldots, G_{k/2-1}, G_2, G_0$. In
other words, the first column of the first row block of $R_3$ contains
elements of $G_1$, the second column of this row block contains
elements of $G_3$, the third column contains $G_4$, and so on. For
$R_4$, we use the partition $F_{k/2-1}, F_{k/2-2}, \ldots, F_1, F_0$.
By placing the four rectangles $R_1, R_2, R_3$, and $R_4$ in the
following configuration, we obtain an $(l+1)k\times k$ rectangle $R$:

$$R = \left[\begin{array}{cc} R_1 & R_2 \\ R_3 & R_4 \end{array}\right]$$

It is easy to see that $R$ is a Sudoku rectangle. Next, we remove the
last $k-r$ rows of $R$ to obtain an $(m, k, n)$-Sudoku rectangle
$R'$. 

Since the partition $G_1, G_3, G_4, \ldots, G_{k/2-1}, G_2, G_0$ was
used to construct $R_3$, in the first column of this rectangle
elements of these sets appear in this order. Since $R_3$ contains
$l+1-\frac{k}{2}<\frac{k}2$ row blocks, none of the elements of $G_0$
can appear in the first column of $R_3$ (and thereofre in the first
column of $R'$). Similarly, none of the elements of $G_2$ appear in
the last column of $R_3$ (and therefore the $\frac{k}2$th column of
$R'$. The first column of $R_4$ contains elements of $F_{k/2-1},
F_{k/2-2}, \ldots$, in this order. Thus, the first column of $R_4$
(and therefore the $(\frac{k}2+1)$th column of $R'$) does not contain
any of the elements of $F_0$. In the second column of $R_4$, we have
elements of $F_{k/2-2}, F_{k/2-3}, \ldots$, in this order. If $l < k -
2$, none of the elements of $F_0$ appear in this column (and therefore
the $(\frac{k}2+2)$th column of $R'$). If $l = k - 2$, elements of
$F_0$ appear in the last row block in the second column of $R_4$. But
then since $R'$ is obtained from $R$ by removing its last $k-r > 0$
rows, the {\em last} element of $F_0$ does not appear in the
$(\frac{k}2+1)$th column of $R'$. Let $x$ denote the last element of
$F_0$, $x_1$ denote the last element of $G_0$, and $x_2$ denote the
last element of $G_2$. By the above arguments, we know that $x$ does
not appear in $(\frac{k}2+1)$th and $(\frac{k}2+2)$th columns of $R'$,
$x_1$ does not appear in its first column, and $x_2$ does not appear
in its $\frac{k}2$th column. Also, by the construction of $R_1$ and
$R_2$, $x$ appears in the first column of the first row block and the
$\frac{k}2$th column of the second row block of $R'$, $x_1$ appears in
the $(\frac{k}2+1)$th column of the first row block of $R'$, and $x_2$
appears in the $(\frac{k}2+2)$th column of the second row block of
$R'$. Finally, since $x_1$, and $x_2$ are the last elements of the
sets $G_0$ and $G_2$, they can only appear in the last row in any row
block in $R$. Therefore, since $R'$ is obtained from $R$ by removing
the last $k-r \ge 1$ rows, neither $x_1$ nor $x_2$ appear in the last
row block of $R'$.

We now perturb the rectangle $R'$ to obtain another rectangle that is
not completable. We do this by first swapping $x$ and $x_1$ in the
first row block of $R'$, and $x$ and $x_2$ in the second row block of
$R'$. By the above observations, these swaps do not violate the Sudoku
conditions. Furthermore, after these swaps, the $(\frac{k}2+1)$th
column of the rectangle does not contain $x_1$ and its
$(\frac{k}2+2)$th column does not contain $x_2$. Therefore, we can
change the element in row $lk + 1$ and column $(\frac{k}2+1)$ to
$x_1$, and the element in row $lk + 1$ and column $(\frac{k}2+2)$ to
$x_2$. Again, it is easy to see that Sudoku conditions are
preserved. Let $R''$ denote the resulting $(m, k, n)$-Sudoku
rectangle.

We argue that $R''$ is not completable. Consider the first $k/2$
columns of $R''$. To complete $R''$, we must place $\frac{k}2(k-r)$
distinct elements in rows $m+1, \ldots, m+k-r = (l+1)k$ of these
columns. Let $T$ denote the set of these elements. Every element in
$F\setminus\{x\}$ appears in each of the first $k/2$ columns of
$R''$. Therefore, $T$ cannot contain any elements of $F$, except
possibly $x$. Also, it cannot contain any of the elements of $G$ that
appear in the last row block of $R''$. There are precisely
$r\cdot\frac{k}2 + 2$ such elements, namely the $r\cdot\frac{k}2$
elements of $G$ that appear in the first $k/2$ columns of the last row
block of $R''$, and $x_1$ and $x_2$. This means that there are at most
$\frac{k}2(k-r)-2$ elements of $G$ that can be in $T$. Therefore, $T$
cannot contain more than $\frac{k}2(k-r)-1$ elements, which is a contradiction.

\item $l\ge k/2$ and $k$ is odd: The high-level idea of the
  construction is similar to the one in case (b): we construct an
  $(l+1)k \times k$ rectangle $R$ by composing four smaller
  rectangles, truncate it to get an $m\times k$ rectangle $R'$, and
  then slightly perturb it to get an $m\times k$ rectangle $R''$ that
  is not completable.

Let $F = \{1,\ldots,k\lceil\frac{k}2\rceil\}$, $G =
\{k\lceil\frac{k}2\rceil + 1,\ldots,k^2\}$, and $H=\{k^2 + 1,\ldots,
k^2 + k\}$.  Note that $|F| = k\lceil\frac{k}2\rceil$, $|G| =
k\lfloor\frac{k}2\rfloor$, and $|G\cup H| = k\lceil\frac{k}2\rceil$.
We construct the following Sudoku rectangles using Lemma~\ref{lem:2}:

\begin{itemize}
\item a
  $(k\lceil\frac{k}{2}\rceil, \lfloor\frac{k}2\rfloor, n)$-Sudoku rectangle $R_1$ with
  entries in $F$, 
\item a $(k\lceil\frac{k}{2}\rceil, \lceil\frac{k}2\rceil, n)$-Sudoku rectangle $R_2$
  with entries in $G\cup H$.
\item an $((l+1-\lceil\frac{k}{2}\rceil)k, \lfloor\frac{k}2\rfloor, n)$-Sudoku rectangle $R_3$
  with entries in $G$, and
\item an $((l+2-\lceil\frac{k}{2}\rceil)k, \lceil\frac{k}2\rceil, n)$-Sudoku rectangle $R_4$
  with entries in $F$.
\end{itemize}
Recall that the construction in Lemma \ref{lem:2} starts by defining a
partition of the set of elements into subsets of size $k$. Let
$F_0,F_1,\ldots,F_{\lceil k / 2 \rceil-1}$ be the partition of $F$
used to construct $R_1$, and $H, G_0,G_1,\ldots, G_{\lfloor k / 2
  \rfloor-1}$ be the partition of $G\cup H$ used to construct
$R_2$. Without loss of generality, we assume that the partition
$G_{\lfloor k / 2 \rfloor-1}, G_{\lfloor k / 2 \rfloor-2}, \ldots,
G_0$ is used to construct $R_3$, and the partition $F_{\lceil k / 2
  \rceil-1}, F_{\lceil k / 2 \rceil-2}, \ldots, F_1, F_0$ is used to
construct $R_4$. This means that the first row block of $R_1$ contains
elements of $F_0$ in its first column, $F_1$ in the second column, and
$F_{\lceil k / 2\rceil - 2}$ in its last (i.e., $\lfloor k /
2\rfloor$th) column. Therefore, the first row block of $R_1$ does not
contain any of the elements of $F_{\lceil k / 2\rceil -
  1}$. Similarly, the second row block of $R_1$ does not contain any
element in $F_0$, the third row block does not contain any element in
$F_1$, and so on. We now modify $R_2$ as follows: Replace elements of
$H$ in the first row block of $R_2$ by elements of $F_{\lceil k /
  2\rceil - 1}$, the elements of $H$ in the second row block of $R_2$
by elements of $F_0$, and so on. Let $R'_2$ denote the resulting
rectangle. All entries of $R'_2$ are in $F\cup G$. Note that by the
construction in Lemma~\ref{lem:2}, elements of $H$ appear in the first
column of $R_2$ in the first row block, the last (i.e., $\lceil k /
2\rceil$th) column in its second row block, \ldots, and the second
column in its last (i.e., $\lceil k / 2\rceil$th) row block. This
means that the set of elements appearing in the first column of $R'_2$
is $G\cup F_{\lceil k / 2\rceil - 1}$, in its second column is $G\cup
F_{\lceil k / 2\rceil - 2}$, and so on. On the other hand, the set of
elements appearing in the $i$'th column of the first row block of
$R_4$ is precisely $F_{\lceil k / 2\rceil - i}$. Therefore, if we
remove the first row block of $R_4$, we obtain an
$((l+1-\lceil\frac{k}{2}\rceil)k, \lfloor\frac{k}2\rfloor, n)$-Sudoku 
rectangle $R'_4$ with elements in $F$ such that the $i$'th
column of $R'_4$ does not contain any of the elements of $F_{\lceil k
  / 2\rceil - i}$. We are now ready to define an $(l+1)k\times k$
rectangle $R$ by placing the four rectangles $R_1, R'_2, R_3$, and
$R'_4$ as follows:

$$R = \left[\begin{array}{cc} R_1 & R'_2 \\ R_3 & R'_4 \end{array}\right]$$

Given the above construction, it is easy to see that $R$ is a
Sudoku rectangle with elements in $F\cup G$. Let $R'$ denote the
$(m, k, n)$-Sudoku rectangle obtained by removing the last $k-r$
rows of $R$. The last step is to perturb $R'$ by moving a few elements
in such a way that the resulting rectangle is not completable.

Recall that the partition $G_{\lfloor k / 2 \rfloor-1}, G_{\lfloor k /
  2 \rfloor-2}, \ldots, G_0$ was used to construct $R_3$. Therefore,
in the first column of $R_3$, we have the elements in $G_{\lfloor k /
  2 \rfloor-1}, G_{\lfloor k / 2 \rfloor-2}, \ldots, G_s$, where $s =
\lfloor k / 2 \rfloor-(l+1-\lceil\frac{k}{2}\rceil) = k - 1 - l >
0$. Therefore, none of the elements in $G_0$ appears in the first
column of $R_3$, and therefore in the first column of $R'$. Similarly,
none of the elements of $G_{\lfloor k / 2 \rfloor-1}$ appears in the
second column of $R'$. In the last column of $R_4$, we have the
elements of $F_0, F_{\lceil k / 2 \rceil-1}, F_{\lceil k / 2
  \rceil-2}, \ldots, F_{s+1}$. Therefore, none of the elements of
$F_1$ appears in the last column of $R'$. In the first column of
$R_4$, we have elements of $F_{\lceil k / 2 \rceil-1}, F_{\lceil k / 2
  \rceil-2}, \ldots, F_{s}$. Therefore, since $R'$ is obtained by
removing the last $k-r \ge 1$ rows of $R$, the {\em last} element of
$F_1$ does not appear in the $\lceil k / 2\rceil$th column of $R'$
(which corresponds to the first column of $R_4$). Now, let $x$ denote
the last element of $F_1$, $x_1$ denote the last element of
$G_{\lfloor k / 2 \rfloor-1}$, and $x_2$ denote the last element of
$G_0$. To summarize, the above arguments, we know that $x$ does not
appear in neither the last nor the $\lceil k / 2\rceil$th column of
$R'$, $x_1$ does not appear in the second column of $R'$, and $x_2$
does not appear in the first column of $R'$. Furthermore, by the
definition of $R_1$ and $R_2$, $x$ appears in the second column of the
first row block of $R'$ and the first column of the second row block
of $R'$ (where all elements of $F_1$ are listed), $x_1$ appears in the
last column of the first row block of $R'$ (where all the elements of
$G_{\lfloor k / 2 \rfloor-1}$ are listed), and $x_2$ appears in the
$\lceil k / 2\rceil$th column of the second row block of $R'$ (where
all the elements of $G_0$ are listed). Finally, since $x_1$ and $x_2$
are the last elements of $G_{\lfloor k / 2 \rfloor-1}$ and $G_0$, they
can only appear in the last row in any row block in $R$. Therefore,
since $R'$ is obtained from $R$ by removing the last $k-r \ge 1$ rows,
neither $x_1$ nor $x_2$ appear in the last row block of $R'$.

We are now ready to perform the swaps that would make $R'$ not
completable. In the first row block of $R'$, we swap $x$ and $x_1$
(which appear in the second and the last columns). In the second row
block, we swap $x$ and $x_2$ (occuring in the first and the $\lceil k
/ 2\rceil$th column). By the above observations, these swaps do not
violate the Sudoku conditions.  After these swaps, there is no $x_1$
in the last column of the rectangle. Therefore, we can replace the
element in the row $lk + 1$ of this column by $x_1$. Similarly, we can
replace the element in the row $lk + 1$ of the $\lceil k / 2\rceil$th
column by $x_2$. Let $R''$ denote the resulting $(m, k, n)$-Sudoku
rectangle.

We now prove that $R''$ is not completable. To see this, consider the
first $\lfloor k / 2\rfloor$ columns, and notice that to complete
$R''$, we need to place $(k-r)\lfloor k / 2\rfloor$ distinct elements
in rows $m + 1$ through $m + k - r$ of these columns. Let $T$ denote
the set of these elements.  Since all elements in $F\setminus\{x\}$
appear in each of the first $\lfloor k / 2\rfloor$ columns of $R''$,
we must have $T\subseteq G\cup\{x\}$. On the other hand, by the block
condition, none of the elements in $T$ must appear in the last row
block of $R''$. This row block contains precisely
$r\lfloor\frac{k}2\rfloor$ elements of $G$ in its first $\lfloor k /
2\rfloor$ columns and two elements of $G$ (namely, $x_1$ and $x_2$) in
the remaining columns. Therefore, $T$ cannot contain more than
$(k-r)\lfloor k / 2\rfloor - 1$ distinct elements, which is a
contradiction.
\end{enumerate}
\end{proof}

\section{Algorithms for Sudoku Rectangle Completion}
\label{sec:alg}

The two-stage procedure used in the proof of Theorem~\ref{thm:thirdcase} is essentially
a polynomial-time algoritm for completing a Sudoku rectangle into a Sudoku square. 
Furthermore, we note that the same procedure applies for every value of $m$, and not just
for those for which Theorem~\ref{thm:thirdcase} guarantees the existence of a Sudoku square.
This is a simple observation, and follows from the fact that if the given Sudoku rectangle $R$ has
a valid extension, then this extension gives a perfect 1-to-$(k-r)$ matching in the graph constructed
in the first stage. On the other hand, if $R$ is not completable, then the graph constructed in the 
first stage does not have any perfect 1-to-$(k-r)$ matching, since if it did, this matching can be 
turned into an extension to a Sudoku square using the method outlined in the proof of 
Theorem~\ref{thm:thirdcase}. Therefore, we have the following result.

\begin{theorem}
There is a polynomial time algorithm that given an $m\times n$ Sudoku
rectangle $R$, finds a completion of $R$ to a Sudoku square, or
decides that $R$ is not completable to a Sudoku square.
\end{theorem}

An interesting follow-up question is whether the above result can be generalized to the case
where $R$ is a general partial Sudoku square or when $R$ is a $(p, q, n)$-Sudoku rectangle. 
We conjecture that at least in the former case, the problem is NP-complete.

\section{Asymptotics of the Number of Sudoku Squares}
\label{sec:counting}
Let $\sud(n)$ denote the number of Sudoku squares of order $n$. In
this section, we give a tight asymptotic bound on
$(\sud(n))^{1/n^2}$. Our proof uses the Sudoku rectangle completion
procedure of the Section~\ref{sec:characterization} as well as classical bounds on the
permanent of matrices. Our result is summarized in the following
theorem.

\begin{theorem}
\label{thm:count}
We have $(\sud(n))^{1/n^2} \sim e^{-3}n$ as $n\rightarrow\infty$.
\end{theorem}

This should be contrasted with a similar result for Latin squares
(see, for example,~\cite[Theorem 17.3]{MR1871828}), which shows that
$(L(n))^{1/n^2}\sim e^{-2}n$ as $n\rightarrow\infty$, where $L(n)$ is
the number of Latin squares of order $n$. In other words,

\begin{corollary}
The fraction of the Latin squares of order $n$ that are Sudoku
squares is $(\frac1e + o(1))^{n^2}$.
\end{corollary}

Proof of Theorem~\ref{thm:count} uses two classical results about the
permanent of matrices. The first result, known as the Van der Waerden
conjecture (proved independently by Falikman and Egoritsjev, see~\cite{MR1871828}), states
that the permanent of every doubly stochastic matrix of order $n$ is
at least $n!n^{-n}$. This implies

 \begin{theorema}{ \rm (\cite{MR1871828})}
\label{thm:A}
The number of perfect matchings in an $r$-regular bipartite graph with
$n$ vertices in each part is at least $n!(r/n)^n$.
\end{theorema}

The other result is an upper bound on the permanent of matrices. This
result, conjectured by Minc in 1967 and proved by Br\'egman in 1973,
shows that the permanent of a $(0,1)$-matrix with row-sums $r_1, r_2,
\ldots, r_n$ is at most $\prod_j (r_j!)^{1/r_j}$. This implies

 \begin{theorema}{ \rm (\cite{MR1871828})}
\label{thm:B}
The number of perfect matchings in an $r$-regular bipartite graph with
$n$ vertices in each part is at most $(r!)^{n/r}$.
\end{theorema}

\def\m{\mathfrak{m}}
\def\M{\mathfrak{M}}

Let $\m(n,r)$ and $\M(n,r)$ denote the minimum and the maximum of the
number of perfect matchings in an $r$-regular bipartite graph with
$n$ vertices in each part. Theorems \ref{thm:A} and \ref{thm:B} can be
written as 

$$n!(r/n)^n \le \m(n,r) \le \M(n,r)\le (r!)^{n/r}.$$

We are now ready to prove Theorem~\ref{thm:count}.

\begin{proof}[Proof of Theorem~\ref{thm:count}]
We use the Sudoku rectangle completion procedure of  Section~\ref{sec:characterization}
to build a Sudoku square row-block by row-block, and count the number
of ways each row-block can be completed. Assume we are at a stage that
$l-1$ row-blocks are completed, and we want to count the number of
ways the $l$th row-block can be completed.

To complete the $l$th row-block, we first process the blocks in this
row-block one by one, and for each block, solve a matching problem to
assign an unordered set of $k$ elements to each column, in such a way
that each element is assigned to only one column within the block, and
an element that is assigned to a column does not appear in the first
$(l-1)$ row-blocks in that column. In the proof of
Theorem~\ref{thm:fullblocks}, we did this by formulating an
assignment problem with $n$ vertices (corresponding to the elements)
on one side and $k$ vertices (corresponding to the columns) on the
other side, and finding an assignment that assigns $k$ elements to
each column. To count the number of such assignments, we formulate
this as a matching problem by replacing each column vertices by $k$ copies.
The assignments in the original graph correspond to perfect matchings
in this new graph. This graph has $n$ vertices on each side and is
$(n-k(l-1))$-regular. Therefore, the number of perfect matchings in
this graph is bounded from below and above by $\m(n, n-k(l-1))$ and
$\M(n, n-k(l-1))$. For each column, since the ordering of which element
is assigned to which copy of that column graph is not important, the number should be
divided by $k!$. This means that the total number of assignments of
$k$ elements to each column, for each block in this row-block, is
between

$$\frac{\m(n, n-k(l-1))}{(k!)^k} \qquad \mbox{and}\qquad \frac{\M(n,
  n-k(l-1))}{(k!)^k}.$$

The total number of assignments for all blocks of this row-block is
the above value raised to the power $k$.

Next, we count the number of ways each of the $k$ elements assigned to
the column can be assigned to the $k$ rows of this column in a way
that no row contains repeated elements. We do this row by row, by
constructing a matching problem for each row. For the first row of the
row-block, the graph consists of $n$ nodes corresponding to the $n$
elements on one side and $n$ nodes corresponding to the $n$ columns on
the other side. Each column node is connected by an edge to each of
the $k$ elements assigned to it in the first stage. It is easy to see
that this is a $k$-regular graph. Each perfect matching in this graph
gives a way to assign elements to the first row of the $l$th
row-block. The number of such perfect matchings is between $\m(n,k)$
and $\M(n,k)$. For the next row, we have the same graph, except the
element already assigned to the first row is no longer connected to
the corresponding column. The graph is $(k-1)$-regular, and has
between $\m(n,k-1)$ and $\M(n,k-1)$ perfect matchings. Continuing this
process, we show that the number of ways to assign the elements to
rows in this row-block is between

$$\m(n,k)\cdot\m(n,k-1)\cdots \m(n,1) \qquad \mbox{and}\qquad
\M(n,k)\cdot\M(n,k-1)\cdots \M(n,1).$$

To obtain a minimum and maximum for the total number of Sudoku
squares, we multiply the above values for all $l=1,\ldots,k$. This
gives the following bounds on the total number of Sudoku squares of order $n$:

\begin{equation}
\label{eqn:lb}
\sud(n)\ge \prod_{l=1}^k\left[\frac{\m(n, n-k(l-1))}{(k!)^k}
\right]^k\cdot \left(\prod_{r=1}^k \m(n,r)\right)^k.
\end{equation}

\begin{equation}
\label{eqn:ub}
\sud(n)\le \prod_{l=1}^k\left[\frac{\M(n, n-k(l-1))}{(k!)^k}
\right]^k\cdot \left(\prod_{r=1}^k \M(n,r)\right)^k.
\end{equation}

Using Theorem~\ref{thm:A} and inequality \eqref{eqn:lb}, we have:

\begin{eqnarray*}
\sud(n) & \ge &
\prod_{l=1}^k\left[\frac{n! (n-k(l-1))^n}{k!^kn^n}
\right]^k\cdot \left(\prod_{r=1}^k (n! (r/n)^n)
\right)^k
\\
&=&
\frac{n!^{n}}{k!^{kn}\cdot k^{n^2}}\prod_{l=1}^k(k-(l-1))^{kn}
\cdot \frac{n!^n}{n^{n^2}} \prod_{r=1}^k r^{kn}\\
&=&
\frac{n!^{n}k!^{kn}}{k!^{kn}\cdot k^{n^2}}
\cdot \frac{n!^nk!^{kn}}{n^{n^2}}
\\
&=&
\frac{n!^{2n}k!^{kn}}{k^{n^2}n^{n^2}}.
\end{eqnarray*}

Using Stirling's formula, we have $n! > (n/e)^n$ and $k! >
(k/e)^k$. Therefore,

$$\sud(n)\ge \frac{n^{2n^2}e^{-2n^2}k^{n^2}e^{-n^2}}{k^{n^2}n^{n^2}}
= (n/e^3)^{n^2}.
$$

Similarly, using Theorem~\ref{thm:B} and inequality \eqref{eqn:ub}, we have:

\begin{eqnarray*}
\sud(n) & \le &
\prod_{l=1}^k\left[\frac{(n-k(l-1))!^{n/(n-k(l-1))}}{(k!)^k}
\right]^k\cdot \left(\prod_{r=1}^k r!^{n/r}\right)^k.
\end{eqnarray*}

We use the following extension of Stirling's formula~\cite{MR0228020}
that holds for every $x$:

$$x! < (x/e)^x\sqrt{2\pi x}\cdot e^{1/(12 x)}.$$

For $x\ge 1$, $\sqrt{2\pi}\cdot e^{1/(12 x)} < 3$, and therefore, $x! <
3(x/e)^x\sqrt{x}$. Using this inequality, we have:

\begin{eqnarray*}
\sud(n) & \le &
\frac1{k!^{kn}}\prod_{l=1}^k(n-k(l-1))!^{kn/(n-k(l-1))}\cdot \prod_{r=1}^k r!^{kn/r}
\\ & < &
\frac1{k!^{kn}}\prod_{l=1}^k\left(3\left(\frac{n-k(l-1)}e\right)^{n-k(l-1)}\sqrt{n-k(l-1)}\right)^{kn/(n-k(l-1))}
\cdot \prod_{r=1}^k \left(3(r/e)^r\sqrt{r}\right)^{kn/r}
\\ & = &
\frac1{k!^{kn}}\prod_{l=1}^k\left(\frac{k(k-(l-1))}e\right)^{kn}
\cdot \prod_{l=1}^k\left(3\sqrt{n-k(l-1)}\right)^{n/(k-(l-1))}
\cdot \prod_{r=1}^k \left(r/e\right)^{kn}
\cdot \prod_{r=1}^k \left(3\sqrt{r}\right)^{kn/r}
\\ & = &
\frac{k^{n^2}k!^{kn}}{k!^{kn}e^{n^2}}
\cdot \prod_{r=1}^k\left(3\sqrt{rk}\right)^{n/r}
\cdot \frac{k!^{kn}}{e^{n^2}}
\cdot \prod_{r=1}^k \left(3\sqrt{r}\right)^{kn/r}
\\ & = &
\frac{k^{n^2}k!^{kn}}{e^{2n^2}}
\cdot \prod_{r=1}^k\left((9r)^{1+k}k\right)^{n/(2r)}
\\ & < &
\frac{k^{n^2}(3(k/e)^k\sqrt{k})^{kn}}{e^{2n^2}}
\cdot \prod_{r=1}^k\left((9r)^{1+k}k\right)^{n/(2r)}
\\ & = &
\frac{n^{n^2}}{e^{3n^2}}\cdot(3\sqrt{k})^{kn}
\cdot \prod_{r=1}^k\left((9r)^{1+k}k\right)^{n/(2r)}.
\end{eqnarray*}

Therefore,

\begin{equation}
\label{eqn:sud:1}
(\sud(n))^{1/n^2}  <
\frac{n}{e^3}\cdot(3\sqrt{k})^{1/k}
\cdot \prod_{r=1}^k\left((9r)^{1+k}k\right)^{1/(2rn)}.
\end{equation}

All that remains is to prove that $(3\sqrt{k})^{1/k} \cdot
\prod_{r=1}^k\left((9r)^{1+k}k\right)^{1/(2rn)} = 1+o(1)$. We use the
inequality that for $x < 1$, $e^x < 1+2x$. Therefore,

\begin{equation}
\label{eqn:sud:2}
(3\sqrt{k})^{1/k} = e^{\frac{\ln(3\sqrt{k})}{k}}< 1 + \frac{2\ln(3\sqrt{k})}{k} = 1 + o(1).
\end{equation}

Also, using the bound $\sum_{r=1}^k \frac1r < 1 + \ln(k)$ on the
harmonic series, we obtain

\begin{eqnarray}
\label{eqn:sud:3}
\prod_{r=1}^k\left((9r)^{1+k}k\right)^{1/(2rn)}
& < &
\prod_{r=1}^k\left((9k)^{1+k}k\right)^{1/(2rn)} \nonumber\\
& < &
\prod_{r=1}^k\left((9k)^{2k}\right)^{1/(2rn)} \nonumber\\
& = &
(9k)^{\frac{1}{k}\sum_{r=1}^k\frac1r}
\nonumber\\
& < &
(9k)^{\frac{1+\ln(k)}{k}}
\nonumber\\
& = &
\exp\left(\frac{1+\ln(k)}{k}\ln(9k)\right)
\nonumber\\
& < & 1 + 2\cdot\frac{1+\ln(k)}{k}\ln(9k) = 1+o(1).
\end{eqnarray}

Putting together Equations~\eqref{eqn:sud:1}, \eqref{eqn:sud:2}, and
\eqref{eqn:sud:3}, we obtain:

$$(\sud(n))^{1/n^2}  < \frac{n}{e^3} (1+o(1)),$$
as desired.
\end{proof}
\section*{Acknowledgements}
Part of the research of the second author was done while he was visiting professor Delaram Kahrobaee at the
Graduate Center of City University of New York. He thanks for the hospitality during his visit.




\end{document}